\documentclass[10pt]{amsart}
\usepackage{latexsym,enumerate}
\usepackage{amsmath,amsthm,amsopn,amstext,amscd,amsfonts,amssymb,dsfont}
\usepackage[latin1]{inputenc}
\usepackage[active]{srcltx}

\newtheorem{Theorem}{Theorem}[section]
\newtheorem{Corollary}[Theorem]{Corollary}
\newtheorem{Definition}[Theorem]{Definition}
\newtheorem{Lemma}[Theorem]{Lemma}
\newtheorem{Proposition}[Theorem]{Proposition}
\newtheorem{remarks}[Theorem]{Remarks}
\newtheorem{remark}[Theorem]{Remark}
\newtheorem{example}[Theorem]{Example}
\newtheorem{examples}[Theorem]{Examples}

\begin{document}

\title[\hfilneg \hfil Convergence of series of strongly integrable random variables]
{Convergence of series of strongly integrable random variables and
applications}
\author[F. Boukhari and D. Malti\hfil \hfilneg]
{Fakhreddine Boukhari and Dounyazed Malti}

\address{Fakhreddine Boukhari\hfill\break
Department of mathematics\hfill\break Faculty of sciences, Abou
Bekr Belkaid University \hfill\break BP 119, Tlemcen 13000.
Algeria} \email{f\_boukhari@yahoo.fr}

\address{Dounyazed Fatiha Malti\hfill\break
Sciences and technology preparatory school \hfill\break BP 165, RP
Bel horizon Tlemcen 13000. Algeria}
\email{dounyazed.malti@gmail.com}

\date{}

%\thanks{Submitted November 26, 2004. Published May 11, 2005.}
\thanks{Research partially supported by CNEPRU C00L03UN130120150006,
Algeria and \hfill\break\indent The Laboratoire de Statistique et
Mod\'elisations Al\'eatoires.}  \keywords{Almost sure convergence; maximal inequalities;
subgaussian random
 \hfill\break\indent variables; exponential integrability; random series}

%%%%%%%%%%%%%%%%%%%%%%     ABSTRACT

 \begin{abstract}
We investigate the convergence of series of random variables with
second exponential moments. We give sufficient conditions for the
convergence of these series with respect to an exponential Orlicz
norm and almost surely. Applying this result to $d$-subgaussian
series, we examine the asymptotic behavior of weighted series of
subgaussian random variables in a unified setting.
\end{abstract}

\maketitle \numberwithin{equation}{section}
%\newtheorem{theorem}{Theorem}[section]
%\newtheorem{definition}[theorem]{Definition}
%\newtheorem{remark}[theorem]{Remark}

%%%%%%%%%%%%%%%%%%%%%%%%%%%%%%%%%%%%%%%%%%%%%%%%%%%%%%%%%%%%%%%%%%%%%%

%                          Première  PARTIE

%%%%%%%%%%%%%%%%%%%%%%%%%%%%%%%%%%%%%%%%%%%%%%%%%%%%%%%%%%%%%%%%%%%%%%

\section{Introduction}

Let $\tilde{a}=\{a_k,\ k\ge 1\}$ be a sequence of real numbers,
$\mathcal{X}=\{X_k,\ k\ge 1\}$ a sequence of real random
variables, defined on some probability space $(\Omega,
\mathcal{F}, P)$. We focus on the asymptotic behavior of the
weighted series $\mathcal{S}(\tilde{a},\mathcal{X})=\sum_n a_n
X_n$, when the increments belong to some exponential Orlicz space.
This problem was examined for the first time by Chow \cite{chow}
when $\mathcal{X}$ is a sequence of independent subgaussian random
variables. In Azuma \cite{az} the series
$\mathcal{S}(\tilde{a},\mathcal{X})$ were investigated when the
increments form a sequence of conditionally subgaussian random
variables. These results were extended to $m$-dependent
subgaussian random variables by Ouy \cite{ouy} and to negatively
dependent subgaussian random variables by Amini et al
\cite{amini1,amini2}. More recently Guiliano et al \cite{anto2}
examined the convergence of the series
$\mathcal{S}(\tilde{a},\mathcal{X})$ when $\mathcal{X}$ is an
$m$-acceptable sequence of $\phi$-subgaussian random variables,
they obtained positive results assuming that
$x\mapsto\phi(|x|^{1/p})$ is a convex function for some $p\in
[1,2]$, then they deduced the corresponding results for the
classical subgaussian case.

In this note, we first study the asymptotic behavior of the series
$\mathcal{S}=\{\sum_1^n X_k,\ n\ge 1\}$, when $\mathcal{X}$ is a
sequence of random variables satisfying some increment condition
with respect to an exponential Orlicz norm, defined by the
function $\varphi(x)=\exp{(x^2)}-1$. Using the entropy criterion
of Dudley, we prove our main result in section 2. When the series
$\mathcal{S}$ converges, we provide estimates for the Orlicz norms
of the limit and the two-sided maximal function of its partial
sums. We would like to underline the fact that the result of
section 2 is stated without any assumption on the dependence
structure of $\mathcal{X}$.

Subgaussian random variables are typical examples of random
variables belonging to the Orlicz space $L^{\varphi}(\Omega)$. We
collect some useful facts about these variables in section 3, we
also introduce the class of processes with $d$-subgaussian
increments. This last notion will be crucial in our investigation.

In section 4 the problem of convergence of the series
$\mathcal{S}(\tilde{a},\mathcal{X})$ is considered assuming that
$\mathcal{X}$ is a $d$-subgaussian increments sequence. This
framework enables us to deduce corresponding results for series of
independent or negatively dependent increments and for
conditionally subgaussian series, we thus recover some results of
\cite{amini1,amini2,az,chow, anto2}. Finally, we examine an
example of a $d$-subgaussian series, which is beyond the scope of
the last quoted papers.

%%%%%%%%%%%%%%%%%%%%%%%%%%%%%%%%%%%%%%%%%%%%%%%%%%%%%%%%%%%%%%%%%%%%%%%%%%%%%%%%%%%%%%%%%%%
%
%                                  Deuxième Partie
%
%%%%%%%%%%%%%%%%%%%%%%%%%%%%%%%%%%%%%%%%%%%%%%%%%%%%%%%%%%%%%%%%%%%%%%%%%%%%%%%%%%%%%%%%%%%

\section{Convergence of $L^{\varphi}$-valued series under increment condition}

Let $(\Omega,\mathcal{F}, P)$ be a probability space, a Young
function $\psi:\ \mathbb{R}_+\longrightarrow\mathbb{R}_+$ is an
increasing convex function with $\psi(0)=0$ and
$\psi(x)\rightarrow +\infty$ when $x\rightarrow +\infty$. For any
Young function $\psi$ we can associate the Orlicz space
$L^\psi(\Omega)$, the space of random variable
$X:\Omega\longrightarrow\mathbb{R}$ such that
$\mathbb{E}\bigl(\psi(a|X|)\bigr)<\infty$ for some $0<a<\infty$,
see for instance  Chapter III in \cite{rr}. Recall that the space
$L^\psi(\Omega)$ is endowed with the norm
$$
\forall X\in L^\psi(P),\quad\|X\|_\psi=\inf\Bigl\{c>0\ :\
\mathbb{E}\bigl(\psi\{\frac{|X|}{c}\}\bigr)\le 1\Bigr\}
$$

\noindent and that $\bigl(L^\psi(\Omega), \|\cdot\|_\psi\bigr)$ is
a Banach space.

Let $E$ be a non empty set, a pseudo-metric $d$ is a positive
application on $E\times E$ which has the properties of a distance
except that it does not necessarily separate points ($d(s,t) = 0$
does not always imply $s=t$). For a pseudo-metric space $(E,d)$
and $\delta>0$, the $\delta$-entropy number $N(E,d,\delta)$ is the
smallest covering number (possibly infinite) of $E$ by open
$d$-balls of radius $\delta$. Dudley \cite{dud} used this last
notion to formulate his famous entropy criterion for local
continuity and boundedness of sample paths of Gaussian processes.
This result was further extended by Pisier to non-Gaussian
processes satisfying incremental Orlicz conditions, see for
instance \cite{lt}, Chapter XI. Since we only deal with random
series, we state Dudley's criterion for countably indexed
processes.

\begin{Theorem}[Dudley]\label{dudley}
Let $E$ be a countable set, provided with a pseudo-metric $d$ and
let $X=\{X_t,\ t\in E \}$ be a stochastic precess indexed on $E$,
with a basic probability space $(\Omega,\mathcal{F}, P)$, and
satisfying the following increment condition
\begin{equation*}
\forall s,t\in E,\qquad \|X_t-X_s\|_{\psi}\le d(s,t).\label{accr0}
\end{equation*}

\noindent Set $diam(E,d)=\sup_{s,t\in E}d(s,t)$ and assume that
the entropy integral

\begin{equation*}
\mathcal{I}(E,d)=\int_0^{diam(E,d)} \psi^{-1}\bigl(N(E,d,u)\bigr)\
du\label{dud0}
\end{equation*}

\noindent is convergent. Then
\begin{equation}
\bigl\|\sup_{s,t\in E}|X_t-X_s|\bigr\|_{\psi}\le 8\ \!
\mathcal{I}(E,d).\label{accr01}
\end{equation}
\end{Theorem}

\begin{remark}
Theorem \ref{dudley} is  stated and proved in \cite{w1} with a
universal constant $C$ in \eqref{accr01}.
\end{remark}

From now on,  $\varphi$ will denote the Young function defined on
$\mathbb{R}_+$ by $\varphi(x)=\exp(x^2)-1,\ L^{\varphi}(\Omega)$
the corresponding Orlicz space. Consider a sequence of real random
variables $\mathcal{X}=\{X_k,\ k\ge 1\}$ in $L^{\varphi}(\Omega)$
and the sequence of their partial sums $\mathcal{S}=\{S_n,\ n\ge
1\}$, that is $S_n=\sum_{k=1}^n X_k$ for $n\ge 1$. In this section
we are interested in the asymptotic behavior of the series
$\mathcal{S}$ when the sequence $\mathcal{X}$ satisfies the
following increment condition:\smallskip

\textsl{For some sequence $\tilde{u}=\{u_k,\ k\ge 1\}$ of positive
real numbers, and some $\alpha>0$, we have}
\begin{equation}
\forall\ \! 0\le n<m<\infty,\qquad\|\sum_{k=n+1}^m
X_k\|_{\varphi}\le\Bigl(\sum_{k=n+1}^m
u_k\Bigr)^{\alpha}\label{accr1}
\end{equation}
%\textsl{for some} $\alpha>0$.\smallskip

We emphasize that this will be the only condition on $\mathcal{X}$
for the validity of our main result (Theorem \ref{pr1}), in
particular there is no restriction on the dependence structure of
the underlying sequence.
\begin{Theorem}\label{pr1}
Let $\{X_k,\ k\ge 1\}$ be a sequence of random variables in
$L^{\varphi}(\Omega)$ satisfying the increment condition
\eqref{accr1}, then
\begin{equation}
\Bigl\|\max_{n<s<t\le m}|\sum_{k=s+1}^t X_k|\Bigr\|_{\varphi}\le 8
C(\alpha)\Bigl(\sum_{k=n+1}^{m} u_k\Bigr)^{\alpha},\label{sup2}
\end{equation}
\noindent for $0\le n< m<\infty$, where
$C(\alpha)=\frac{2^{2\alpha+2}}{\sqrt{\alpha}}
\int_{\sqrt{\alpha\ln{3}}}^{+\infty} x^2 e^{-x^2}\ dx<\infty$ .

If in addition the series $\sum_{k\ge 1} u_k$ converges, then
\eqref{sup2} holds true for $m=\infty$, the sequence of partial
sums $\mathcal{S}=\{S_n,\ n\ge 1\}$ converges in
$L^{\varphi}(\Omega)$ and $P$-almost surely to a random variable
$S_{\infty}\in L^{\varphi}(\Omega)$ such that
$\|S_{\infty}\|_{\varphi}\le (\sum_{k\ge 1} u_k)^{\alpha}$,
furthermore
\begin{equation}
\Bigl\|\sup_{1\le i<j}|\sum_{k=i+1}^j X_k|\Bigr\|_{\varphi}\le
8C(\alpha)\Bigl(\sum_{k=1}^{\infty}
u_k\Bigr)^{\alpha}.\label{sup1}
\end{equation}
\end{Theorem}

\begin{proof}
We first prove \eqref{sup1}, let $\alpha>0$, assume
$\textbf{u}:=\sum_{k\ge 1} u_k<\infty$ and set $S_0=0,\ U_0=0,\
U_n=\sum_{k=1}^n u_k$ for $n\ge 1$ and $\mathcal{U}=\{U_n,\ n\ge
0\}$. We will apply Theorem \ref{dudley} with $E=\mathbb{N}$,
endowed with the pseudo metric
$d_{\varphi}(n,m)=\|S_m-S_n\|_{\varphi}$ for $n\le m$, to succeed
we need to estimate the entropy numbers
$N(\mathbb{N},d_{\varphi},\epsilon)$. Let $d$ be the distance
defined on $\mathbb{N}^2$ by $d(i,i)=0$ for each $i\in \mathbb{N}$
and $d(i,j)=\sum_{k=i\wedge j+1}^{i\vee j} u_k$ otherwise. Using
$d_{\varphi}$ and $d$, assumption \eqref{accr1} becomes
$d_{\varphi}(n,m)\le d^{\alpha}(n,m)$ for all $n,m$ in
$\mathbb{N}$. We also have $\sup_{s,t\ge 0} d(s,t)=
\textbf{u}<\infty$, hence $N(\mathbb{N},d,\epsilon)<\infty$ for
any $\epsilon>0$, this entails the existence of a finite sequence
$\{l_k,\ 1\le k\le N(\mathbb{N},d,\epsilon)\}$ in $\mathbb{N}$,
and a finite collection of open $d$-balls $\{B_d(l_k,\epsilon),\
1\le k\le N(\mathbb{N},d,\epsilon)\}$ which form an
$\epsilon$-covering of $\mathbb{N}$. Thus
\begin{equation*}
\mathbb{N}\subset\bigcup_{k=1}^{N(\mathbb{N},d,\epsilon)}
B_d(l_k,\epsilon)\subset\bigcup_{k=1}^{N(\mathbb{N},d,\epsilon)}
B_{d_{\varphi}}(l_k,\epsilon^{\alpha}),
\end{equation*}
\noindent this shows that
\begin{equation}
N(\mathbb{N},d_{\varphi},\epsilon^{\alpha})\le
N(\mathbb{N},d,\epsilon).\label{entropie00}
\end{equation}
Let $a,b$ be two reals with $a<b,\ d_0$ the usual distance on
$\mathbb{R}$ ($d_0(x,y)=|x-y|$ for $x,y$ in $\mathbb{R}$) and
denote by  $[x]$ the greatest integer less than or equal to the
real $x$. It is well-known that $N([a,b],d_0,\delta)\le
[\frac{b-a}{\delta}]+1\le 2\ \!\frac{b-a}{\delta}$, for any
$0<\delta<b-a$, this and the fact that $\mathcal{U}$ is a subset
of $[0, \textbf{u}]$, yield
\begin{equation}
N(\mathcal{U},d_0,\delta)\le N([0, \textbf{u}], d_0,\delta)\le 2\
\!\frac{\textbf{u}}{\delta},\label{entropie1}
\end{equation}
\noindent for any $0<\delta<\textbf{u}$. Now $\mathcal{U}$ is a
countable set and $d(n,m)=d_0(U_n,U_m)$, so we can assign to each
open $d$-ball $B_d(n,\delta)$ of $(\mathbb{N},d)$, one and only
one open $d_0$-ball $B_{d_0}(U_n,\delta)$ of $(\mathcal{U},d_0)$
such that: $ m\in B_d(n,\delta)$ if and only if $U_m\in
B_{d_0}(U_n,\delta)$. Combining this with \eqref{entropie1}, we
obtain
\begin{equation*}
N(\mathbb{N},d,\delta)=N(\mathcal{U},d_0,\delta)\le 2\
\!\frac{\textbf{u}}{\delta} \qquad\text{for }\
0<\delta<\textbf{u}.\label{entropie2}
\end{equation*}
\noindent Thus taking $\delta=\epsilon^{1/\alpha}$ in this last
estimate and making use of \eqref{entropie00}, we get
\begin{equation*}
N(\mathbb{N},d_{\varphi},\epsilon)\le 2\ \!\textbf{u}\
\!\epsilon^{-\frac{1}{\alpha}} \qquad\text{for }\
0<\epsilon<\textbf{u}^{\alpha}.\label{entropief}
\end{equation*}
This inequality is enough to ensure the finiteness of the entropy
integral $\mathcal{I}(\mathbb{N},d_{\varphi})$ in Dudley's
Theorem, indeed it is easily seen that
$\varphi^{-1}(x)=\sqrt{\ln(x+1)}$ for $x>0$ and
$diam(\mathbb{N},d_{\varphi})=\sup_{n,m\ge
1}d_{\varphi}(n,m)\le\textbf{u}^{\alpha}$. On the other hand, as
$x>\frac{x+1}{2}$ for any $x>1$, we may write
$$
\mathcal{I}(\mathbb{N},d_{\varphi})\le
\int_0^{\textbf{u}^{\alpha}} \sqrt{\ln\bigl(2\ \!\textbf{u}\
\!\epsilon^{-\frac{1}{\alpha}}+1\bigr)}\ d\epsilon\le
\Bigl[\frac{2^{2\alpha+2}}{\sqrt{\alpha}}
\int_{\sqrt{\alpha\ln{3}}}^{+\infty} x^2 e^{-x^2}\ dx\ \Bigr]\
\!\textbf{u}^{\alpha}:=C(\alpha)\ \!\textbf{u}^{\alpha}
$$
\noindent with  $C(\alpha)=\frac{2^{2\alpha+2}}{\sqrt{\alpha}}
\int_{\sqrt{\alpha\ln{3}}}^{+\infty} x^2 e^{-x^2}\ dx<\infty$.
Combining this with \eqref{accr01}, we reach that
$$
\Bigl\|\sup_{i,j\ge 1}|\sum_{k=i}^j X_k|\Bigr\|_{\varphi}\le
8C(\alpha)\textbf{u}^{\alpha},
$$
\noindent which is exactly \eqref{sup1}. To prove \eqref{sup2},
fix $0\le n<m<\infty$ and apply the first step to the sequences
$\{Y_k,\ k\ge 1\}$ and $\{v_k,\ k\ge 1\}$, where we have set
$Y_k=X_k,\ v_k=u_k$ for $n< k\le m$ and $Y_k=0,\ v_k=0$ otherwise.
It is clear that the increment condition \eqref{accr1} is
fulfilled for these new sequences, moreover $\sum_{k\ge
1}v_k=\sum_{k=n+1}^m u_k<\infty$, hence one has by \eqref{sup1}
\begin{align*}
\Bigl\|\max_{n<s< t\le m}|\sum_{k=s+1}^t X_k|\Bigr\|_{\varphi}
&=\Bigl\|\sup_{1\le i<j}|\sum_{k=i+1}^j Y_k|\Bigr\|_{\varphi}     \\
&\le 8C(\alpha)\Bigl(\sum_{k=1}^{\infty} v_k\Bigr)^{\alpha}       \\
& =  8C(\alpha) \Bigl(\sum_{k=n+1}^{m} u_k\Bigr)^{\alpha}.      \,
\end{align*}
 Now we turn to the convergence problem of the series
$\mathcal{S}$. Since $\sum_{k\ge 1}u_k<\infty$, then \eqref{accr1}
immediately gives $\|\sum_{k=n+1}^m X_k\|_{\varphi}\longrightarrow
0$ as $n,m\rightarrow\infty$, thereby $\mathcal{S}$ is a Cauchy
sequence in $L^{\varphi}(\Omega)$, so it converges in
$\|\cdot\|_{\varphi}$-norm to a random variable $S_{\infty}\in L
^{\varphi}(\Omega)$, moreover taking $n=0$ and letting $m$ tend to
infinity in \eqref{accr1}, we get $\|S_{\infty}\|_{\varphi}\le
\textbf{u}^{\alpha}$. For the almost sure convergence, let $n,r\ge
1$, making use of Markov's inequality, the fact that
$\|\cdot\|_1\le \|\cdot\|_{\varphi}$ and \eqref{sup2}, we obtain
$$
P\bigl(\ \sup_{1\le k\le r}\
\!\bigl|S_{n+k}-S_n\bigr|>\epsilon\bigr)
\le\frac{1}{\epsilon}\Bigl\|\ \!\sup_{n<j\le n+r}|\!\sum_{i=n+1}^j
X_i|\Bigr\|_{\varphi} \le
\frac{8C(\alpha)}{\epsilon}\Bigl(\sum_{k=n+1}^{\infty}
u_k\Bigr)^{\alpha},
$$
\noindent for any $\epsilon> 0$. Consequently $P\bigl(\
\!\sup_{k\ge 1}\
\!\bigl|S_{n+k}-S_n\bigr|>\epsilon\bigr)\longrightarrow 0$ as
$n\rightarrow\infty$, which proves the almost sure convergence of
$\mathcal{S}$ and achieves the proof of Theorem \ref{pr1}.
\end{proof}

\begin{remarks}
For $\alpha>0$, $C(\alpha)$ will denote the positive constant of
Theorem \ref{pr1}, in particular $C(1/2)\simeq 8.26$ and
\begin{equation*}
C(\alpha)\le\sqrt{\frac{2}{\alpha}}\Bigl[(\frac{4}{3})^{\alpha}
\sqrt{2\alpha\ln{3}}+4^{\alpha}\sqrt{\frac{\pi}{2}}\
\Bigr],\label{constante}
\end{equation*}
\end{remarks}

%%%%%%%%%%%%%%%%%%%%%%%%%%%%%%%%%%%%%%%%%%%%%%%%%%%%%%

%                Troisième Partie

%%%%%%%%%%%%%%%%%%%%%%%%%%%%%%%%%%%%%%%%%%%%%%%%%%%%%%

\section{processes with $d$-subgaussian increments}

Subgaussian random variables were introduced by Kahane
\cite{kahane1} in his study of random Fourier series. (see also
\cite{kahane2}, Chapter VI) Later on the class of subgaussian
random variables was studied thoroughly by Buldygin and Kozachenko
\cite{buld0}.

A real random variable $X$ is called subgaussian if there exists
$c\ge 0$ such that
\begin{equation*}
\forall t\in\mathbb{R},\qquad \mathbb{E}\exp\ \!\bigl\{
tX\bigr\}\le\exp\ \!\Bigl\{\frac{c^2 t^2}{2}\Bigr\}\label{stand}.
\end{equation*}
Centered Gaussian random variables and centered bounded random
variables are examples of subgaussian random variables.  Thus we
can easily build subgaussian random variables by truncation and
mean-centering.

 Following \cite{buld}, we denote by $Sub(\Omega)$
the class of subgaussian random variables, defined on some
probability space $(\Omega,\mathcal{F}, P)$. For any $X\in
Sub(\Omega)$, we can associate its subgaussian standard $\tau(X)$
defined by
$$
\tau(X)=\inf\Bigl\{c\ge 0:\ \mathbb{E}\exp\ \!\bigl\{
tX\bigr\}\le\exp\ \!\Bigl\{\frac{c^2 t^2}{2}\Bigr\},\
t\in\mathbb{R}\ \Bigr\}.
$$

It is worth noting that the application $X\longrightarrow\tau(X)$
induces a norm on $Sub(\Omega)$, moreover

\begin{Theorem}[\cite{buld}, Theorem 1.2]\label{banach}
The space $Sub(\Omega)$ is a Banach space with respect to the norm
$\tau(\cdot)$.
\end{Theorem}

The space of subgaussian random variables $Sub(\Omega)$ is a
subspace of Orlicz space $L^{\varphi}(\Omega)$ of section 2. To be
more precise, if $L_1^0(\Omega)$ stands for the space of
integrable and centered random variables, then
$Sub(\Omega)=L^{\varphi}(\Omega)\cap L_1^0(\Omega)$. Furthermore
Kozachenko and Ostrovsky showed that $\tau(\cdot)$ and
$\|\cdot\|_{\varphi}$ are equivalent on $Sub(\Omega)$. Giuliano
\cite{anto1} strengthened their result by proving that
\begin{equation}
\frac{1}{2\sqrt{2}}\
\!\tau(X)\le\|X\|_{\varphi}\le\sqrt{2+2\sqrt{2}}\
\!\tau(X),\label{equivalence}
\end{equation}
\noindent for any $X\in Sub(\Omega)$.

\textsc{Notation :}\ \textsl{For a sequence $\mathcal{X}=\{X_k,\
k\ge 1\}$ in $Sub(\Omega)$, we will denote by
${\bf\tau}=\{\tau_k,\ k\ge 1\}$ the sequence of its subgaussian
standards, i.e. $\tau_k=\tau(X_k)$ for $k\ge 1$, for such random
variables we will write $X_k\hookrightarrow Sub(\tau_k)$}. We also
set $S_n=\sum_{k=1}^n X_k$ for $n\ge 1$.

\begin{Definition}[\cite{lt}, p.\ \!322]\label{dsubgauss}
Let $(T,d)$ be a pseudo-metric space, $\mathcal{Z}=\{Z_t,\ t\in
T\}$ a stochastic precess in $Sub(\Omega)$.
\begin{itemize}
\item The process $\mathcal{Z}$ is said to have $d$-subgaussian
increments, if for all  $\lambda\in \mathbb{R}$ and all $u,v\in T$
\begin{equation}
\mathbb{E}\exp\ \!\bigl\{\lambda(Z_v-Z_u)\bigr\}\le\exp\
\!\Bigl\{\frac{d^2(u,v)\ \! \lambda^2}{2}\Bigr\}.\label{processus}
 \end{equation}
\item A random series $\mathcal{S}=\{S_n,\ n\ge 1\}$ fulfilling
\eqref{processus}, will be called a $d$-subgaussian series. If
$X_1= S_1,\ X_k=S_k-S_{k-1}$ for $k\ge 2$, then
$\mathcal{X}=\{X_k,\ k\ge 1\}$ will be called the $d$-subgaussian
increments sequence of $\mathcal{S}$ or simply the $d$-subgaussian
increments sequence.
\end{itemize}
\end{Definition}

Gaussian processes are certainly prime examples of processes with
$d$-subgaussian increments, however the class of processes with
$d$-subgaussian increments is large. Actually, we can build
examples of such processes using random series of functions, see
for instance Example 5.5 in \cite{jain}. Here we have chosen to
borrow an example from the theory of stochastic integration.
\begin{example}

Let $B=\{B_t,\ t\ge 0\}$ be a standard Brownian motion defined on
$(\Omega,\ \!\mathcal{F}, P)$, $\mathcal{F}_t=\sigma\{B_u,\ u\le
t\}$ its natural filtration and $\mathcal{H}=\{H_t,\ t\ge 0\}$ an
adapted, continuous process, such that $P\{|H_t|\le K\}=1$ for all
$t\ge 0$ and some positive constant $K$. Set for $t\ge 0$
\begin{equation*}
M_t=\int_0^t H_u dB_u\qquad\text{ and }\qquad M=\{M_t,\ t\ge
0\}.\label{yor}
\end{equation*}

$M=\{M_t,\ t\ge 0\}$ is a continuous local martingale and
$<\!\!M,M\!>_t=\int_0^t H^2_u du$ for any $t\ge 0$. Making use of
Itô's Formula, we conclude that the positive process
$$
\mathcal{E}^{\lambda}(M)=\Bigl\{\exp\{\lambda
M_t-\frac{\lambda^2}{2}\int_0^t H^2_u du\},\ t\ge 0\Bigr\}
$$
is also a continuous local martingale, thus it is a
supermartingale. Therefore
$$
\mathbb{E}\bigl(\exp\{\lambda M_t-\frac{\lambda^2}{2}\int_0^t
H^2_u du\}/\mathcal{F}_s\bigr) \le \exp\{\lambda
M_s-\frac{\lambda^2}{2}\int_0^s H^2_u du\},
$$
\noindent for any $0\le s\le t$, which is equivalent to
$$
\mathbb{E}\bigl(\exp\{\lambda(M_t-M_s)-\frac{\lambda^2}{2}\int_s^t
H^2_u du\}/\mathcal{F}_s\bigr) \le 1.
$$
\noindent Now taking expectation in the previous inequality and
using the boundedness assumption on $\mathcal{H}$, we reach that
$$
\mathbb{E}\exp\{\lambda(M_t-M_s)\}\le
\exp\{\frac{K^2(t-s)\lambda^2}{2}\}
$$
\noindent for all $\lambda\in\mathbb{R}$. Hence $M$ is a process
with $d$-subgaussian increments, where we have set
$d^2(s,t)=K^2|t-s|$ for $s,t\ge 0$. On the other hand, if the
bounded process $\mathcal{H}$ is chosen such that for some $t>0$,
$\int_0^t H^2_u du$ is not deterministic, then $M$ can not be a
Gaussian process, see for instance \cite{yor}, Chapter IV.
\end{example}

\begin{remark}\label{independant}

Let $0\le i<j,\ \mathcal{X}=\{X_k,\ k\ge 1\}$ a sequence of random
variables in $Sub(\Omega)$ with $X_k\hookrightarrow Sub(\tau_k)$
for $k\ge 1$. In view of Theorem \ref{banach}, $S_j-S_i\in
Sub(\Omega)$ and $\tau(S_j-S_i)\le \sum_{i+1}^j \tau_k$, thus
\begin{equation*}
\forall t\in\mathbb{R},\qquad \mathbb{E}\exp\
\!\bigl\{t(S_j-S_i)\bigr\}\le\exp\ \Bigl\{\frac{\delta^2(i,j)\ \!
t^2}{2}\Bigr\},
\end{equation*}

\noindent where we have set $\delta(i,i)=0$ and
$\delta(i,j)=\sum_{i\wedge j+1}^{i\vee j} \tau_k$ for $i\neq j$.
This shows that $\{S_n,\ n\ge 1\}$ is a $\delta$-subgaussian
series.
\end{remark}

\section{Convergence of weighted series of subgaussian random variables}

In this section we focus on the convergence problem of series with
$d$-subgaussian increments. The main idea is to rewrite Theorem
\ref{pr1} for these series by means of the equivalence between
$\tau(\cdot)$ and $\|\cdot\|_{\varphi}$ given by
\eqref{equivalence}. For any sequence $\mathcal{Y}=\{Y_k,\ k\ge
1\}$ in $Sub(\Omega),\ S(\mathcal{Y})=\{S_n(\mathcal{Y}),\ n\ge
1\}$ will denote the sequence of its partial sums.  Let $0\le n<
m$ and set
\begin{equation}
S^{*}_{n,m}(\mathcal{Y})= \max_{n<i<j\le m}|\sum_{k=i+1}^j Y_k|\
\quad\text{ and }\quad S^{*}(\mathcal{Y})=\sup_{1\le
i<j}|\sum_{k=i+1}^j Y_k|,\label{max}
\end{equation}
%\noindent The metrics we have encountered in the last section lead
%us to the following

\begin{Definition}\label{distance}
Let $0<\alpha\le 1, {\tilde{u}}=\{u_k,\ k\ge 1\}$  a sequence of
positive real numbers.  $d_{\tilde{u},\alpha}$ is the distance
defined on $\mathbb{N}^2$ by
\begin{equation*}
d_{\tilde{u},\alpha}(i,i)=0\qquad\text{ and }\qquad
d_{\tilde{u},\alpha}(i,j)=\bigl(\sum_{k=i\wedge j+1}^{i\vee j}
u_k\bigr)^{\alpha},\label{metric}
\end{equation*}
\noindent for $i,j\in\mathbb{N}$ with $i\neq j$. When
$u_k=\tau_k^2$ for all $k$ in $\mathbb{N}\backslash\{0\}$, we will
write $d_{\tau^2,\alpha}$.
\end{Definition}

%Applying Theorem \ref{pr1} to $d$-subgaussian series we obtain

\begin{Theorem}\label{pr55}
Let  $0<\alpha\le 1,\ \mathcal{Y}=\{Y_k,\ k\ge 1\}$ a
$d_{\tilde{u},\alpha}$-subgaussian increments sequence, then
\begin{equation*}
\|S^{*}_{n,m}(\mathcal{Y})\|_{\varphi}\le 8\sqrt{2+2\sqrt{2}}\
\!C(\alpha)\ \! d_{\tilde{u},\alpha}(n,m)
\end{equation*}
\noindent for $0\le n< m$, where $C(\alpha)$ is the constant
defined in Theorem \ref{pr1}.

If in addition ${\bf u}:=\sum_{k\ge 1} u_k<\infty$, then
 $\|S^{*}(\mathcal{Y})\|_{\varphi}\le 8\sqrt{2+2\sqrt{2}}\
\!C(\alpha)\ \!{\bf u}^{\alpha}$, furthermore the series
 $\{S_n(\mathcal{Y}),\ n\ge 1\}$
converges with respect to the norms $\tau(\cdot)$ and
$\|\cdot\|_{\varphi}$ and $P$-almost surely to a random variable
$S_{\infty}(\mathcal{Y})\in Sub(\Omega)$ fulfilling
\begin{equation*}
\tau\bigl(S_{\infty}(\mathcal{Y})\bigr)\le {\bf u}^{\alpha}\qquad
\text{ and }\qquad\|S_{\infty}(\mathcal{Y})\|_{\varphi}\le
\sqrt{2+2\sqrt{2}}\ \!{\bf u}^{\alpha}.\label{est1}
\end{equation*}
\end{Theorem}

\begin{proof}

Let $0\le n<m$ and set $u_k^{'}=(2+2\sqrt{2})^{\frac{1}{2\alpha}}\
\! u_k$\ \! for $k\ge 1$, as $S(\mathcal{Y})$ is a sequence with
$d_{\tilde{u},\alpha}$-subgaussian increments, we have
\begin{equation}
\tau\bigl(\sum_{k=n+1}^m Y_k\bigr)\le
d_{\tilde{u},\alpha}(n,m)=\bigl(\sum_{k=n+1}^{m}
u_k\bigr)^{\alpha}.\label{tau1}
\end{equation}
\noindent This and \eqref{equivalence}, yield
$$
\|\sum_{k=n+1}^m Y_k\|_{\varphi}\le\sqrt{2+2\sqrt{2}}\
\!\tau\bigl(\sum_{n+1}^m Y_k\bigr)\le\bigl(\sum_{n+1}^m
u^{'}_k\bigr)^{\alpha}.
$$
\noindent Thus the increment condition \eqref{accr1} of section 2
is fulfilled for $\{Y_k,\ k\ge 1\}$ and $\{u^{'}_k,\ k\ge 1\}$,
the result follows immediately by applying Theorem \ref{pr1}, in
particular $S(\mathcal{Y})$ converges with respect to
$\tau(\cdot)$, to a random variable $S_{\infty}\in Sub(\Omega)$.
Going back to \eqref{tau1}, putting $n=0$ and letting $m$ tend to
infinity we get $\tau\bigl(S_{\infty}(\mathcal{Y})\bigr)\le {\bf
u}^{\alpha}$, this and \eqref{equivalence} allow us to conclude
that $\|S_{\infty}(\mathcal{Y})\|_{\varphi}\le \sqrt{2+2\sqrt{2}}\
\!{\bf u}^{\alpha}$.
\end{proof}

Now we apply Theorem \ref{pr55} in the study of the asymptotic
behavior of weighted series of subgaussian random variables. Let
$\tilde{a}=\{a_k,\ k\ge 1\}$ be a sequence of real numbers,
$\mathcal{X}=\{X_k,\ k\ge 1\}\in Sub(\Omega)$ and
$\mathcal{S}(\tilde{a},\mathcal{X})=\{S_n(\tilde{a},\mathcal{X}),\
n\ge 1\}$, where we have set
\begin{equation*}
S_n(\tilde{a},\mathcal{X})=\sum_1^n a_k X_k,\qquad n\ge
1.\label{som1}
\end{equation*}

\begin{Corollary}\label{delta}

Let $\mathcal{X}$ be a sequence in $Sub(\Omega)$, with
$X_k\hookrightarrow Sub(\tau_k)$ for $k\ge 1$ and assume
$\sum_{k=1}^{\infty} |a_k|\tau_k<\infty$. Then the sequence of
partial sums $\mathcal{S}(\tilde{a},\mathcal{X})$ converges with
respect to the norms $\tau(\cdot)$ and $\|\cdot\|_{\varphi}$ and
$P$-almost surely to $S_{\infty}(\tilde{a},\mathcal{X})\in
Sub(\Omega)$, such that
\begin{equation*}
\tau\bigl(S_{\infty}(\tilde{a},\mathcal{X})\bigr)\le
\sum_{k=1}^{\infty} |a_k|\tau_k\qquad \text{ and
}\qquad\|S_{\infty}(\tilde{a},\mathcal{X})\|_{\varphi}\le
\sqrt{2+2\sqrt{2}}\ \! \sum_{k=1}^{\infty}
|a_k|\tau_k.\label{est1}
\end{equation*}
\end{Corollary}

\begin{proof}
For $k\ge 1$, let $u_k=|a_k|\tau_k$ and $\tilde{u}=\{u_k,\ k\ge
1\}$, then $\tau(a_k X_k)=u_k$, but as $d_{\tilde{u},1}$ coincide
with the metric $\delta$ of Remark \ref{independant}, then
$\mathcal{S}(\tilde{a},\mathcal{X})$ is a
$d_{\tilde{u},1}$-subgaussian series. The conclusion follows by
applying Theorem \ref{pr55} to  $\mathcal{Y}=\{a_k X_k,\ k\ge
1\}$.
\end{proof}

The question raised by the definition of a $d$-subgaussian process
and Theorem \ref{pr55}, is how to find a suitable pseudo-metric
$d$ such that \eqref{processus} holds true? Remark
\ref{independant} shows that a series in $Sub(\Omega)$ is always a
$\delta$-subgaussian series, however when some knowledge of the
dependence structure of the underlying sequence is available, then
we can replace the trivial distance $\delta$ by a more accurate
metric. Let us recall some dependence concepts which will be
useful for our purpose, we begin by the notion of negative
dependence, introduced by Lehmann \cite{leh}.

\begin{Definition}

\

\begin{itemize}

\item Let $n\ge 2,\ x_1,x_2,\ldots, x_n\in\mathbb{R}$. The real
random variables $X_1, X_2,\ldots, X_n$ are said to be negatively
dependent (ND) if
$$
P\Bigl\{\bigcap_{k=1}^n\{X_k\le x_k\}\Bigr\}\le\prod_{k=1}^n
P\{X_k\le x_k\}\quad\text{and}\quad
P\Bigl\{\bigcap_{k=1}^n\{X_k>x_k\}\Bigr\}\le\prod_{k=1}^n
P\{X_k>x_k\}.
$$

\item A sequence $\mathcal{X}=\{X_k,\ k\ge 1\}$ is said to be
negatively dependent (ND) if any finite subfamily $X_{i_1},
X_{i_2},\ldots, X_{i_n}$ is negatively dependent.
\end{itemize}
\end{Definition}

The second notion we need was introduced by Azuma \cite{az}, let
${\bf c}=\{c_k,\ k\ge 1\}$ be a sequence of positive real numbers,
$\mathcal{B}=\{\mathcal{B}_k,\ k\ge 0\}$ an increasing family of
sub $\sigma$-algebras of $\mathcal{F}$, such that
$\mathcal{B}_0=\{\emptyset, \Omega\}$.
\begin{Definition} Let $\mathcal{X}=\{X_k,\ k\ge 1\}$ be a sequence
of random variables, we will say that $\mathcal{X}$ is
$(\mathcal{B},{\bf c}^2)$-conditionally subgaussian if it is a
sequence of martingale differences with respect to $\mathcal{B}$
and
$$
\forall t\in \mathbb{R},\qquad
\mathbb{E}\bigl(\exp\{tX_k\}/B_{k-1}\bigr)\le\exp\
\!\Bigl\{\frac{c_k^2 t^2}{2}\Bigr\}\qquad P.\ \! a.\ \!
s\quad\text{for}\ k\ge 1.
$$
\end{Definition}
For a sequence $\mathcal{X}\in Sub(\Omega)$ enjoying of one of the
aforementioned properties, we have
\begin{Proposition}\label{pr31}
Let $\mathcal{X}=\{X_k,\ k\ge 1\}$ be a sequence in $Sub(\Omega)$,

\begin{enumerate}

\item If $\mathcal{X}$ is a sequence of independent random
variables with $X_k\hookrightarrow Sub(\tau_k)$ for $k\ge 1$, then
$S(\tilde{a},\mathcal{X})$ is a
$d_{\tilde{u},\frac{1}{2}}$-subgaussian series, where
$\tilde{u}=\{a_k^2\tau_k^2,\ k\ge 1\}$.

\item If $\mathcal{X}$ is a sequence of ND random variables with
$X_k\hookrightarrow Sub(\tau_k)$ for $k\ge 1$, then
$S(\tilde{a},\mathcal{X})$ is a
$d_{\tilde{v},\frac{1}{2}}$-subgaussian series, where
$\tilde{v}=\{2a_k^2\tau_k^2,\ k\ge 1\}$.

\item If $\mathcal{X}$ is $(\mathcal{B},{\bf c}^2)$-conditionally
subgaussian, then $S(\tilde{a},\mathcal{X})$ is a
$d_{\tilde{w},\frac{1}{2}}$-subgausian series, where
$\tilde{w}=\{a_k^2 c_k^2,\ k\ge 1\}$.
\end{enumerate}
\end{Proposition}

\begin{proof}
Let $0\le n<m,\ t\in\mathbb{R}$ and $\mathcal{X}$ in
$Sub(\Omega)$.

1. Assume that $\mathcal{X}$ is a sequence of independent random
variables, then we have by Lemma 1.7 of Chapter I in \cite{buld}:
$\tau^2(\sum_{n+1}^m a_k X_k)\le \sum_{n+1}^{m} a_k^2\tau^2_k$,
hence
\begin{equation*}
\mathbb{E}\exp\ \!\bigl\{t\sum_{n+1}^m a_k X_k\bigr\}\le\exp\
\!\Bigl\{\frac{d^2(n,m)\ \! t^2}{2}\Bigr\},
\end{equation*}
\noindent where we have set $d^2(i,j)=\sum_{i\wedge j+1}^{i\vee j}
a_k^2\tau_k^2$ for $i,j\in\mathbb{N}$. This proves the first
point.

2. If $\mathcal{X}$ is a sequence of ND random variables, then
according to the proof of Theorem 1 in \cite{amini2},
$\tau^2(\sum_{n+1}^m a_k X_k)\le 2\sum_{n+1}^{m} a_k^2\tau^2_k$,
which is equivalent to the assertion in the second point.

 3. Let $\mathcal{X}$ be $(\mathcal{B},{\bf c}^2)$-conditionally subgaussian,
 then
\begin{align*}
\mathbb{E}\exp\Bigl\{t\!\!\sum_{k=n+1}^m a_kX_k\Bigr\} & =
\mathbb{E}\Bigl(\mathbb{E}\bigl(\exp\Bigl\{t\!\!\sum_{k=n+1}^{m}
a_kX_k\Bigr\}/
\mathcal{B}_{m-1}\bigr)\Bigr)   \\
&= \mathbb{E}\Bigl(\exp\Bigl\{t\!\!\sum_{k=n+1}^{m-1}
a_kX_k\Bigr\}
\mathbb{E}\bigl(\exp\{t a_m X_m\}/\mathcal{B}_{m-1}\bigr)\Bigr)   \\
& \le\exp\{\frac{a_m^2c^2_m
t^2}{2}\}\mathbb{E}\exp\Bigl\{t\!\!\sum_{k=n+1}^{m-1}
a_kX_k\Bigr\}. \,
\end{align*}
\noindent By induction, we conclude that
$$
\mathbb{E}\exp\Bigl\{t\!\!\sum_{k=n+1}^m a_kX_k\Bigr\}\le
\exp\Bigl\{\frac{\sum_{k=n+1}^m a_k^2c^2_k}{2}\ \! t^2\Bigr\},
$$
\noindent which yields the last item of Proposition \ref{pr31}.
\end{proof}

\begin{remark}
Proposition \ref{pr31} and Theorem \ref{pr55} highlight the
interest of introducing sequences with
$d_{\tilde{u},\alpha}$-subgaussian increments, it shows that this
class of random variables provides a unified framework to study
the asymptotic behavior of series in $Sub(\Omega)$.
\end{remark}

Let $S^{*}_{n,m}(\tilde{a},\mathcal{X})$ and $S^{*}(\tilde{a},
\mathcal{X})$ be the maximal functions of
$\mathcal{S}(\tilde{a},\mathcal{X})$ defined by setting $Y_k=a_k
X_k$ in \eqref{max}.
\begin{Corollary}\label{acceptable2}
Let $0\le n<m,\ \mathcal{X}=\{X_k,\ k\ge 1\}$ a ND sequence in
$Sub(\Omega)$, with $X_k\hookrightarrow Sub(\tau_k)$ for $k\ge 1$,
then
\begin{equation*}
\|S^{*}_{n,m}(\tilde{a},\mathcal{X})\|_{\varphi}\le
16\sqrt{1+\sqrt{2}}\ \!C(1/2)\ \!\bigl(\sum_{k=n+1}^m
a_k^2\tau^2_k\bigr)^{\frac{1}{2}},\label{a1}
\end{equation*}
\noindent where $C(1/2)\simeq 8.26$.  Further, if
\begin{equation}
A^2(\tilde{a},\tau):=\sum_{k\ge 1}
a_k^2\tau^2_k<\infty,\label{acceptable3}
\end{equation}
 \noindent then
 $\|S^{*}(\tilde{a}, \mathcal{X})\|_{\varphi}\le 16\sqrt{1+\sqrt{2}}\
\!C(1/2)\ \! A(\tilde{a},\tau)$, the series $\{S_n(\tilde{a},
\mathcal{X}),\ n\ge 1\}$ converges with respect to the norms
$\tau(\cdot)$ and $\|\cdot\|_{\varphi}$ and $P$-almost surely to a
random variable $S_{\infty}(\tilde{a}, \mathcal{X})\in
Sub(\Omega)$ fulfilling
\begin{equation*}
\tau\bigl(S_{\infty}(\tilde{a}, \mathcal{X})\bigr)\le \sqrt{2}
A(\tilde{a},\tau)\qquad \text{ and }\qquad\|S_{\infty}(\tilde{a},
\mathcal{X})\|_{\varphi}\le 2\sqrt{1+\sqrt{2}}\ \!
A(\tilde{a},\tau).\label{est1}
\end{equation*}
\end{Corollary}

\begin{proof}
By the second item of Proposition \ref{pr31}, $S(\tilde{a},
\mathcal{X})$ is a $d_{\tilde{u},\frac{1}{2}}$-subgaussian series
with $\tilde{u}=\{2a^2_k\tau_k^2,\ k\ge 1\}$. The result follows
by applying Theorem \ref{pr55} to $\{a_kX_k,\ k\ge 1\}$ with ${\bf
u}=2 A^2(\tilde{a},\tau)$ and $\alpha=1/2$.
\end{proof}
For series with conditionally subgaussian increments sequences, we
have
\begin{Corollary}\label{condi}
Let $0\le n<m,\ \mathcal{X}=\{X_k,\ k\ge 1\}$ a $(\mathcal{B},{\bf
c}^2)$-conditionally subgaussian, then
\begin{equation*}
\|S^{*}_{n,m}(\tilde{a},\mathcal{X})\|_{\varphi}\le
8\sqrt{2+2\sqrt{2}}\ \!C(1/2)\bigl(\sum_{k=n+1}^m a_k^2
c^2_k\bigr)^{\frac{1}{2}},\label{a2}
\end{equation*}
\noindent where $C(1/2)\simeq 8.26$. If in addition
$B^2(\tilde{a},{\bf c}):=\sum_{k\ge 1} a_k^2 c^2_k<\infty$, then
$\|S^{*}(\tilde{a}, \mathcal{X})\|_{\varphi}\le
8\sqrt{2+2\sqrt{2}}\ \!C(1/2)B(\tilde{a},{\bf c})$,
 furthermore $\{S_n(\tilde{a},
\mathcal{X}),\ n\ge 1\}$ converges with respect to the norms
$\tau(\cdot)$ and $\|\cdot\|_{\varphi}$ and $P$-almost surely to a
random variable $S_{\infty}(\tilde{a}, \mathcal{X})$ in
$Sub(\Omega)$ and fulfilling
\begin{equation*}
\tau\bigl(S_{\infty}(\tilde{a}, \mathcal{X})\bigr)\le
B(\tilde{a},{\bf c})\qquad \text{ and
}\qquad\|S_{\infty}(\tilde{a}, \mathcal{X})\|_{\varphi}\le
\sqrt{2+2\sqrt{2}}\ \! B(\tilde{a},{\bf c}).\label{est2}
\end{equation*}
\end{Corollary}

\begin{proof}

By the last item in Proposition \ref{pr31}, $S(\tilde{a},
\mathcal{X})$ is a $d_{\tilde{u},\frac{1}{2}}$-subgaussian series
with $\tilde{u}=\{a^2_k c_k^2,\ k\ge 1\}$, hence it is enough to
apply Theorem \ref{pr55} to $\{a_kX_k,\ k\ge 1\}$ with ${\bf
u}=B^2(\tilde{a},{\bf c})$ and $\alpha=1/2$.
\end{proof}

\begin{remarks}
\begin{enumerate}
\item Corollary \ref{acceptable2} extend the well-known results on
the convergence of series of independent gaussian $($resp.
Rademacher$)$ random variables.

\item Let $\sigma=\{\sigma_k,\ k\ge 1\}$ be a sequence of positive
numbers and consider a sequence $g=\{g_k,\ k\ge 1\}$ of
independent Gaussian random variables such that
$g_k\hookrightarrow\mathcal{N}(0,\sigma_k^2)$. Put
$S_n(\tilde{a},g)=\sum_{1}^n a_k g_k$ for $n\ge 1$. In this case
$\tau(g_k)=\sigma_k$, besides by Theorem 6.1 in \cite{lt}, the
series $\{S_n(\tilde{a},g),\ n\ge 1\}$ converges $P$ almost surely
and with respect to the $L^2$-norm, if and only if
$\sum_1^{\infty}a_k^2\tau^2(g_k)<\infty$. This indicates that the
condition in \eqref{acceptable3} can not be weakened without
additional assumptions.

\item Corollary \ref{acceptable2} $($resp. Corollary
\ref{condi}$)$ is stated in \cite{amini1, amini2, chow, ouy}
$($resp. \cite{az}$)$ under stronger conditions.
\end{enumerate}
\end{remarks}

%%%%%%%%%%%%%%%%%%%%%%%%%%%%%%%%%%%%%%%%%%%%%%%%%

The following result illustrates again the convenience of
introducing $d$-subgaussian processes. Indeed the increments
sequence of the series in Theorem \ref{weber} is not necessarily
acceptable or $(\mathcal{B}, {\bf c}^2)$-conditionally
subgaussian, therefore the results of \cite{amini1, amini2, az,
chow, anto2, ouy} do not apply for it. We first remind that the
decoupling coefficient $p(g)$ of a stationary centered Gaussian
sequence $g=\{g_k,\ k\ge 1\}$ is given by
$$
p(g):=\sum_{k=1}^{\infty}\Bigl|
\frac{\mathbb{E}(g_1g_k)}{\mathbb{E}(g_1^2)}\Bigr|
$$
\begin{Theorem}\label{weber}
Let $g=\{g_k,\ k\ge 1\},\ g^{'}=\{g^{'}_k,\ k\ge 1\}$ be two
stationary Gaussian sequences of $\mathcal{N}(0,1)$ random
variables with finite decoupling coefficients, $\tilde{a}=\{a_k,\
k\ge 1\}$ and $\tilde{b}=\{b_k,\ k\ge 1\}$ two sequences of real
numbers. For $n\ge 1$, put
$$
R_n=\sum_{k=1}^n a_kg_k+b_kg^{'}_k,\qquad \mathcal{R}=\{R_n,\ n\ge
1\}.
$$
Then, $\mathcal{R}$ is a sequence with $\tilde{d}$-subgaussian
increment sequence, where we have set
\begin{equation}
\tilde{d}^2(i,j)=2\max\bigl(p(g),p(g^{'})\bigr) \sum_{k=i\wedge
j+1}^{i\vee j}\!\!a^2_k+b^2_k,\label{w2}
\end{equation}
\noindent for $i,j\in \mathbb{N}$. If in addition
$\sum_{1}^{\infty} a^2_k+b^2_k<\infty$, then $\mathcal{R}$
converges with respect to $\tau(\cdot)$ and $\|\cdot\|_{\varphi}$
and $P$-almost surely to a random variable $R_{\infty}\in
Sub(\Omega)$ such that
\begin{equation}
\tau(R_{\infty})\le \Bigl(2\max\bigl(p(g),p(g^{'})\bigr)
\!\sum_{k=1}^{\infty}\! a^2_k+b^2_k\Bigr)^{\frac{1}{2}}.\label{r}
\end{equation}
\end{Theorem}

\begin{proof}
First notice that $g$ and $g^{'}$ are not assumed to be
independent, so $R_n$ is not necessarily a Gaussian random
variable. To prove the claim, we argue as in \cite{w0}. We begin
by recording a result of Klein-Landau-Shucker $($\cite{kls},
Theorem $1):$ Let $\gamma=\{\gamma_k,\ k\ge 1\}$ be a stationary
Gaussian sequence of $\mathcal{N}(0,1)$ random variables, with a
finite decoupling coefficient $p(\gamma)$, then
\begin{equation}
\Bigl|\mathbb{E}\Bigl(\prod_{k\in J} H_k(\gamma_k)\Bigr)\Bigr|\le
\prod_{k\in J}\bigl\| H_k(\gamma_k)
\bigr\|_{p(\gamma)},\label{kls}
\end{equation}
\noindent for any finite subset $J$ of $\mathbb{N}$ and any family
$\{H_k,\ k\ge 1\}$ of complex-valued Borel-measurable functions,
where $\|\cdot\|_{r}$ stands for the usual $L^{r}$-norm.

Let $1\le n<m$ and $\lambda\in\mathbb{R}$, applying
Cauchy-Schwarz's inequality, we get
\begin{equation}
\mathbb{E}\exp\bigl\{\lambda\bigl(R_m-R_n\bigr)\bigr\}\le
\Bigl(\mathbb{E}\exp\Bigl\{2\lambda\!\!\sum_{k=n+1}^{m}
a_kg_k\Bigr\}\mathbb{E}\exp\Bigl\{2\lambda\!\!\sum_{k=n+1}^{m}
b_kg^{'}_k\Bigr\}\Bigr)^{\frac{1}{2}}.\label{w1}
\end{equation}
Set $Z(g)=\exp\Bigl\{2\lambda\sum_{k=n+1}^{m} a_kg_k\Bigr\},\
J=\{n+1,\ldots,m\}$ and apply \eqref{kls} to $g$ and $g^{'}$ with
$H_k^{\tilde{a}}(x)=e^{2\lambda a_k x}$ and $
H_k^{\tilde{b}}(x)=e^{2\lambda\ \!\! b_k x}$, respectively. We
obtain
$$
\mathbb{E}Z(g)\le\exp\Bigl\{2p(g)\lambda^2
\!\!\sum_{k=n+1}^{m}a^2_k\Bigr\}\quad\text{and}\quad
\mathbb{E}Z(g^{'})\le\exp\Bigl\{2p(g^{'})\lambda^2
\!\!\sum_{k=n+1}^{m}b^2_k\Bigr\},
$$
where we have used $\mathbb{E}\exp\{\lambda\
\!\mathcal{N}(0,1)\}=\exp(\frac{\lambda^2}{2})$. Going back to
\eqref{w1} we reach that
$$
\mathbb{E}\exp\bigl\{\lambda\bigl(R_m-R_n\bigr)\bigr\}\le
\exp\Bigl\{\max\bigl(p(g),p(g^{'})\bigr)
\lambda^2\!\!\sum_{k=n+1}^{m}\! a^2_k+b^2_k\Bigr\},
$$
\noindent thus $\mathcal{R}$ is a sequence with
$\tilde{d}$-subgaussian increments sequence, where $\tilde{d}$ is
the distance defined by \eqref{w2}. Furthermore
$\tau^2\bigl(R_m-R_n\bigr)\le2\max\bigl(p(g),p(g^{'})\bigr)
\sum_{k=n+1}^{m}\! a^2_k+b^2_k$. If in addition $\tilde{a}$ and
$\tilde{b}$ are in $\ell^2$, then Theorem \ref{pr55} ensures that
$\mathcal{R}$ converges with respect to $\tau(\cdot)$ and
$\|\cdot\|_{\varphi}$ and $P$-almost surely to a random variable
$R_{\infty}\in Sub(\Omega)$ and \eqref{r} follows.
\end{proof}
Now let $\tilde{b}=\{b_{nk},\ n\ge 1,\ k\ge 1\}$ be an array of
real numbers, we examine the problem of convergence of
$\mathcal{T}(\tilde{b},\mathcal{X})=\{T_n(\tilde{b},\mathcal{X}),\
n\ge 1\}$, where
\begin{equation}
T_n(\tilde{b},\mathcal{X})=\sum_{k=1}^{\infty}b_{nk} X_k,\qquad
n\ge 1\label{T}
\end{equation}
\noindent and $\mathcal{X}=\{X_k,\ k\ge 1\}$ belongs to one of the
following classes :
$$
\mathcal{X}\ \text{is ND},\quad X_k\hookrightarrow
Sub(\tau_k),\quad\text{ and }\quad
A_n^2(\tilde{b},\tau):=\sum_{k\ge
1}b_{nk}^2\tau_k^2<\infty,\leqno(\mathcal{ND})\label{serie4}
$$
\noindent or
$$
\mathcal{X}\ \text{is}\ (\mathcal{B},{\bf
c}^2)\text{-conditionally subgaussian}\quad\text{ and }\quad
B_n^2(\tilde{b},{\bf c}):=\sum_{k\ge 1}b_{nk}^2
c_k^2<\infty,\leqno(\mathcal{CS})\label{serie5}
$$
\noindent for $n,k\ge 1$. We point out that in view of Corollaries
\ref{acceptable2} and \ref{condi}, the random variable
$T_n(\tilde{b},\mathcal{X})$ is well-defined in both cases.
\begin{Theorem}\label{49}
Let $n\ge 1,\ \mathcal{X}$ a sequence in $(\mathcal{ND})$ $($resp.
in $(\mathcal{CS}))$, then $T_n(\tilde{b},\mathcal{X})$ is in
$Sub(\Omega)$ and $\tau\bigl(T_n(\tilde{b},\mathcal{X})\bigr)\le
\sqrt{2}A_n(\tilde{b},\tau)$ $($resp.
$\tau\bigl(T_n(\tilde{b},\mathcal{X})\bigr)\le B_n(\tilde{b},{\bf
c}))$. Furthermore, if
\begin{equation}
\sum_1^{\infty}\exp\{-\frac{t^2}{4\
\!A^2_n(\tilde{b},\tau)}\}<\infty\qquad (\text{resp.}
\sum_1^{\infty}\exp\{-\frac{t^2}{2\ \!B^2_n(\tilde{b},{\bf
c})}\}<\infty)\label{serie6}
\end{equation}
\noindent for each $t>0$, then the sequence
$\mathcal{T}(\tilde{b},\mathcal{X})$ converges $P$-almost surely
and with respect to $\tau(\cdot)$ and $\|\cdot\|_{\varphi}$ to
zero.
\end{Theorem}

\begin{proof}
Assume that $\mathcal{X}$ is in the class $(\mathcal{ND})$, fix
$n\ge 1$ and set $u_k(n)=2b^2_{nk}\tau^2_k$ and
$\tilde{u}(n)=\{u_k(n),\ k\ge 1\}$. According to the second item
of Proposition \ref{pr31}, $\{\sum_{k=1}^{m}b_{nk} X_k,\quad m\ge
1\}$ is a $d_{\tilde{u}(n),\frac{1}{2}}$-subgaussian series, as
$A_n^2(\tilde{b},\tau)<\infty$, it converges $P$-almost surely and
with respect to $\tau(\cdot)$ and $\|\cdot\|_{\varphi}$ to
$T_n(\tilde{b},\mathcal{X})\in Sub(\Omega)$ and
\begin{equation}
\tau\bigl(T_n(\tilde{b},\mathcal{X})\bigr)\le
\sqrt{2}A_n(\tilde{b},\tau).\label{cond2}
\end{equation}
\noindent Let $\lambda,\ t>0$, the previous estimate and Markov's
inequality yield
\begin{align*}
P\{|T_n(\tilde{b},\mathcal{X})|\ge t\} &\le\exp\{-\lambda
t\}\bigl(\mathbb{E} \exp\{\lambda
T_n(\tilde{b},\mathcal{X})\}+\mathbb{E}
\exp\{-\lambda T_n(\tilde{b},\mathcal{X})\}\bigr)        \\
& \le 2\ \!\exp\{-\lambda t+A^2_n(\tilde{b},\tau)\lambda^2\} . \,
\end{align*}
\noindent Optimizing with respect to $\lambda$, we reach that
$$
P\{|T_n(\tilde{b},\mathcal{X})|\ge t\}\le 2\exp\{-\frac{t^2}{4\
\!A^2_n(\tilde{b},\tau)}\}.
$$
Combining this with \eqref{serie6} and the Borel-Cantelli lemma,
we get the $P$-\! almost sure convergence to zero of
$\mathcal{T}(\tilde{b},\mathcal{X})$. Besides, by \eqref{serie6}
$\lim_{n\rightarrow\infty} A_n(\tilde{b},\tau)=0$, this and
\eqref{cond2} entail the $\tau$-convergence to zero of
$\mathcal{T}(\tilde{b},\mathcal{X})$. The convergence with respect
to $\|\cdot\|_{\varphi}$ comes easily from \eqref{equivalence}.
The proof of the Theorem for a conditionally subgaussian
increments sequence is similar, we omit it.
\end{proof}
Let $\beta>0$, applying the last result to the sequence
$b_{nk}=\bigl(n\ln^{1+\beta}(n)\bigr)^{-\frac{1}{2}}$ for $2\le
k\le n$ and $b_{nk}=0$ otherwise, we get
\begin{Corollary}\label{am}
Let $\mathcal{X}$ be a sequence in $(\mathcal{ND})$ $($resp. in
$(\mathcal{CS}))$ with $\sup_{k\ge 1}\tau_k<\infty\ ($resp.
$\sup_{k\ge 1}c_k<\infty)$, then $P$-almost surely,
$$
\lim_{n\rightarrow\infty}\bigl(n\ln^{1+\beta}(n)\bigr)^{-1/2}S_n=0.
$$
\end{Corollary}

%%%%%%%%%%%%%%%%%%%%%%%%%%%%%%%%%%%%%%%%%%%%%%%%%%%%%%%%%%%%%%%%%%%%%%%%%

%%           R E F E R E N C E S

%%%%%%%%%%%%%%%%%%%%%%%%%%%%%%%%%%%%%%%%%%%%%%%%%%%%%%%%%%%%%%%%%%%%%%%%%

\end{document}